\documentclass[10pt,reqno]{amsart}
\usepackage{amsfonts}
\usepackage{amsaddr}
\usepackage{latexsym}
\usepackage{amsmath}
\usepackage{amssymb}
\usepackage{amsthm}
\usepackage{mathrsfs}
\usepackage[pagebackref,colorlinks=true]{hyperref}

\newenvironment{acknowledgements}
{
  \begin{abstract}
}{
  \end{abstract}
}

\newtheorem{thrm}{Theorem}[section]

\newtheorem{lemma}[thrm]{Lemma}

\newtheorem{prop}[thrm]{Proposition}

\newtheorem{cor}[thrm]{Corollary}

\newtheorem{dfn}[thrm]{Definition}

\newtheorem{rmrk}[thrm]{Remark}

\newtheorem{conv}[thrm]{Convention}

\newtheorem{exam}[thrm]{Example}

\newcommand{\newsection}{  
\setcounter{equation}{0}\section}
\def\appendix#1{\addtocounter{section}{1}\setcounter{equation}{0}
\renewcommand{\thesection}{\Alph{section}}
\section*{Appendix \thesection\protect\indent \parbox[t]{11.15cm}{#1}}
\addcontentsline{toc}{section}{Appendix \thesection\ \ \ #1}}

\newcommand{\be}{\begin{eqnarray}}
\newcommand{\ee}{\end{eqnarray}}
\newcommand{\bea}{\begin{eqnarray}}
\newcommand{\eea}{\end{eqnarray}}
\newcommand{\ba}{\begin{array}}
\newcommand{\ea}{\end{array}}

\newenvironment{frcseries}{\fontfamily{frc}\selectfont}{}

\def\d{\delta}

\def\sb {{\nabla}}
\def\LC{{\nabla^g}}
\def\p{{\varphi}}
\def\ph{{\Phi}}
\def\ps{{\Psi^+}}
\def\sp{{\Psi^-}}

\begin{document}

\title[Curvature identities on $SPIN(7)$--manifold]{The Riemannian curvature  identities for the torsion connection
on  $SPIN(7)$--manifold and generalized Ricci solitons}


\author{Stefan Ivanov}
\address{University of Sofia, Faculty of Mathematics and
Informatics, blvd. James Bourchier 5, 1164, Sofia, Bulgaria}
\address{and Institute of Mathematics and Informatics, Bulgarian Academy of
Sciences} 
\address{e-mail: ivanovsp@fmi.uni-sofia.bg}

\author{Alexander Petkov}
\address{University of Sofia, Faculty of Mathematics and
Informatics, blvd. James Bourchier 5, 1164, Sofia, Bulgaria}
\address{e-mail: a\_petkov\_fmi@abv.bg}

\maketitle

\begin{abstract}
It is shown that on compact $Spin(7)$--manifold with exterior derivative of the Lee form lying in the Lie algebra $spin(7)$ the curvature $R$ of the $Spin(7)$--torsion connection $R\in S^2\Lambda^2$ with vanishing Ricci tensor if and only if the $3$-form torsion is  parallel with respect to the Levi-Civita connection. 
It is also proved that $R$ satisfies the Riemannian first Bianchi identity exactly when  the $3$-form torsion is  parallel with respect to the Levi-Civita and to the $Spin(7)$--torsion connections simultaneously.

Precise conditions for a  compact $Spin(7)$--manifold to has closed torsion are given in terms of the Ricci tensor of the $Spin(7)$--torsion connection. 
It is shown that a  compact $Spin(7)$--manifold with closed torsion is Ricci flat if and only if either the norm of the torsion or the Riemannian scalar curvature  is constant. 
It is proved that any compact $Spin(7)$--manifold with closed torsion 3-form is a generalized gradient Ricci soliton and this is equivalent to a certain vector field to be parallel with respect to the torsion connection. In particular, this vector field preserves 
the $Spin(7)$--structure.
\medskip


\end{abstract}

\begin{acknowledgements}
We would like to thank Jeffrey Streets, Ilka Agricola and the anonymous referee for  extremely useful remarks, comments and suggestions.

The research of S.I.  is partially supported    by Contract KP-06-H72-1/05.12.2023 with the National Science Fund of Bulgaria, Contract 80-10-192 / 17.5.2023   with the Sofia University "St.Kl.Ohridski" and  the National Science Fund of Bulgaria, National Scientific Program ``VIHREN", Project KP-06-DV-7. The research of A.P.  is partially  financed by the European Union-Next Generation EU, through the National Recovery and Resilience Plan of the Republic of Bulgaria, project 
SUMMIT BG-RRP-2.004-0008-C01. 
\end{acknowledgements}

\tableofcontents

\setcounter{section}{0}
\setcounter{subsection}{0}

\newsection{Introduction}
Riemannian manifolds with metric connections having totally skew-symmetric torsion and special holonomy received a lot of interest in mathematics and theoretical physics mainly from supersymmetric string theories and supergravity.  The main reason becomes from the Hull-Strominger system which describes the supersymmetric background in heterotic string theories \cite{Str,Hull}. The number of preserved supersymmetries depends on the number  of  parallel spinors with respect to a metric connection $\sb$ with totally skew-symmetric torsion $T$. The presence of a $\nabla$-parallel spinor leads to restriction of the
holonomy group $Hol(\nabla)$ of the torsion connection $\nabla$  to a group which is a stabilizer of a non-trivial spinor. These groups are known to be   $SU(n)$, $Sp(n)$, $G_2$ and  $Spin(7)$ due to the work of Hitchin \cite{Hit} and Wang \cite{Wang}. 
A detailed analysis of the possible geometries is
carried out in \cite{GMW}. 

The Hull-Strominger system has been investigated intensively for $SU(n)$-holonomy, i.e. on complex non-K\"ahler manifolds by many mathematicians and physicists. The first compact solutions has been constructed  in \cite{LY,yau,yau1} and a geometric flow point of view has been developed  in \cite{PPZ,PPZ1,Ph,PPZ4}.

In dimension 8, the existence of parallel spinors with respect to a metric connection with torsion 3-form  is  important in supersymmetric string theories since the number of parallel spinors determines the number of preserved supersymmetries, which is the first Killing spinor equation in the heterotic Hull-Strominger system in dimension eight \cite{GKMW,GMW,GMPW},  and  
this leads to the reduction of the holonomy group of the torsion connection to a subgroup of $Spin(7)$.  It is shown in \cite{I} that any $Spin(7)$--manifold admits a unique metric connection with totally skew-symmetric torsion  preserving the $Spin(7)$--structure, i.e. the first Killing spinor equation always has a solution (see also \cite{Fr,Mer} for another proof of this fact). 

 For application to the Hull-Strominger system, the $Spin(7)$--manifold should be compact and  the dilatino equation in the Hull-Strominger system leads  the Spin(7)--manifold has to be  globally conformally balanced which means that the Lee form $\theta$  defined below in \eqref{g2li} must be an exact form, $\theta=df$ for a smooth function $f$  which represents the dilaton  \cite{GKMW,GMW,GMPW,MS}.   A geometric flow point of view on the heterotic $Spin(7)$-Hull-Strominger system in dimension eight is developed recently in \cite{AMP}.  
 
 Special attention is also paid when the torsion 3-form is closed, $dT=0$. For example, in type II string theory, $T$ is identified with the 3-form field strength. This is required by construction to satisfy $dT=0$ (see e.g. \cite{GKMW,GMW}).  
More generally, the geometry of a torsion connection with closed torsion form appears  in the frame work of the generalized Ricci flow and the  generalized (gradient) Ricci solitons developed by Garcia-Fernandez and Streets \cite{GFS} (see  the references therein).
 
The main purpose of this paper is to investigate the curvature properties of the torsion connection on 8-dimensional compact $Spin(7)$--manifolds.

First, we observe the following 
\begin{thrm}\label{closT}
Let $(M,\ph)$ be a  compact  $Spin(7)$--manifold.  The torsion 3-form is  closed, $dT=0$ if and only if the Ricci tensor $Ric$ of the  Spin(7)-torsion connection  is given by the following formula
\begin{equation}\label{clos1}Ric=-\frac76\sb\theta.
\end{equation}
 In particular,  the  exterior derivative of the Lee form belongs to $\Lambda^2_{21}\cong spin(7)$ and $\delta T=\frac76d^{\sb}\theta$.
\end{thrm}
This helps to prove the next
\begin{thrm}\label{closTt}
Let $(M,\ph)$ be a  compact  $Spin(7)$--manifold with closed torsion. Then  the following conditions are equivalent:
\begin{itemize}
\item[a)] The Spin(7)-torsion connection  is Ricci flat, $Ric=0$;
\item[b)] The norm of the torsion is constant, $d||T||^2=0$;
\item[c)] The Lee form is co-closed, $\delta\theta=0$;
\item[d)] The Riemannian scalar curvature is constant, $Scal^g=const.$;
\item[e)] The Spin(7)-torsion connection has vanishing scalar curvature, $Scal=0$;
\item[f)] The Lee form is $\nabla$-parallel, $\sb\theta=0$;
\end{itemize}
In each of the six  cases above the torsion is  a harmonic 3-form.
\end{thrm}
We show in Theorem~\ref{inf} that any compact $Spin(7)$--manifold with closed torsion 3-form is a generalized gradient Ricci soliton and this condition is equivalent to a certain vector field to be parallel with respect to the $Spin(7)$--torsion connection. We also find out that this  vector field  is an infinitesimal automorphism of the $Spin(7)$--structure.

Studying the curvature properties of the $Spin(7)$--torsion connection, we find necessary conditions the torsion 3-form to be closed and harmonic in the compact case,
\begin{thrm}\label{mainsu3}
Let $(M,\ph)$ be a compact $Spin(7)$--manifold  and the exterior derivative of the Lee form lies in the Lie algebra $spin(7)$.

The Spin(7)-torsion connection $\sb$   has curvature $R \in S^2\Lambda^2$ with vanishing Ricci tensor,  
\begin{equation}\label{s2l2}
R(X,Y,Z,V)=R(Z,V,X,Y),\qquad Ric(X,Y)=0,
\end{equation}
if and only if the torsion 3-form $T$ is parallel with respest to the Levi-Civita connection, $\LC T=0$.

In particular, the torsion 3-form  is harmonic, $dT=\delta T=0$ and the Lee form  is $\sb$-parallel, $\sb\theta=0$.
\end{thrm}
As a consequence of Theorem~\ref{mainsu3} we obtain
\begin{cor}\label{mainspin}
Let $(M,\ph)$ be a compact  $Spin(7)$--manifold with $d\theta=0$.

The curvature of the $Spin(7)$-torsion connection $\sb$ satisfies \eqref{s2l2} 
if and only if the torsion 3-form $T$ is parallel with respest to the Levi-Civita connection, $\LC T=0.$

In this case the Lee form $\theta$ is $\sb$ and $\LC$-parallel, $\sb\theta=\LC\theta=0$.
\end{cor}
Concerning the Riemannian first Bianchi identity, we have
\begin{thrm}\label{co1}
Let $(M,\ph)$ be a compact $Spin(7)$--manifold 
 and $d\theta\in spin(7)$.

The curvature of $\sb$  satisfies the Riemannian first Bianchi identity \eqref{RB} if and only if  
\hspace{0.1cm}$\LC T=0=\sb T.$
 \end{thrm} 
\begin{cor}\label{co2}
Let $(M,\ph)$ be a compact  $Spin(7)$--manifold and $d\theta=0$.

The curvature of $\sb$  
satisfies the Riemannian first Bianchi identity \eqref{RB} if and only if 
\hspace{0.1cm}$\LC T=0=\sb T.$

In this case the Lee form $\theta$ is $\sb$-parallel and $\LC$-parallel.
\end{cor} 
Note that spaces  with parallel torsion 3-form with respect to the torsion connection are investigated in  \cite{AFer,CMS}   and a large number of examples are given there.

We remark that, in general, metric connections with 
 closed torsion 3-form $T$ are closely connected with the generalized Ricci flow. Namely, the fixed points of the generalized Ricci flow are Ricci flat metric connections with harmonic torsion 3-form, $Ric=dT=\delta T=0$, we refer to the recent book \cite{GFS} and the references given there for  mathematical and physical motivation. In this direction, our results show that a compact $Spin(7)$--manifold with $d\theta\in spin(7)$ (in particular locally conformally balanced $Spin(7)$--manifold)  with Ricci flat $Spin(7)$--torsion connection having curvature $R\in S^2\Lambda^2$ is a fixed point of the corresponding generalized Ricci flow. In particular, if the curvature of the $Spin(7)$--torsion  connection satisfies the Riemannian first Bianchi identity then it is a fixed point of the  generalized Ricci flow  provided $d\theta\in spin(7)$.  Moreover, any compact $Spin(7)$--manifold with closed torsion and either constant Riemannian scalar curvature or constant norm of the torsion  is a fixed point of the  generalized Ricci flow.

\begin{rmrk} We recall \cite[Theorem~4.1]{AFF} which states that an irreducible complete and simply connected Riemannian manifold of dimension bigger or equal to $5$ with $\sb$-parallel and closed torsion 3-form, $\sb T=dT=0$ (which is equivalent to $\sb T=\sigma^T=0$ due to \eqref{dh} and \eqref{sigma} below), is a simple compact Lie group or its dual non-compact symmetric space with biinvariant metric, and, in particular the torsion connection is the flat Cartan connection. In this spirit, our results above imply that  the irreducible complete and simply connected case in Theorem~\ref{co1} and Corollary~\ref{co2} may  occur only on the simple compact Lie group $SU(3)$ since the $Spin(7)$--manifold should be a simple  compact Lie group of dimension eight. 
\end{rmrk}

\begin{conv}\label{conv}
Everywhere in the paper we will make no difference between tensors and the corresponding forms via the metric as well as we will  use Einstein summation conventions, i.e. repeated Latin  indices are summed over up to $8$.

We use the tensor norm of a k-form $\alpha, ||\alpha||^2=\sum_{i_1,\dots,i_k}\alpha_{i_1,\dots,i_k}\alpha_{i_1,\dots,i_k}=\alpha_{i_1,\dots,i_k}\alpha_{i_1,\dots,i_k}$ with respect to an orthonormal basis $\{e_1,\dots,e_n\}$.  Note the difference by a factor $k!$ with $|\alpha|^2$, which is the norm of $\alpha$ as a k-form, $||\alpha||^2=k!|\alpha|^2$.
\end{conv}

\section{Preliminaries}
In this section we recall some known curvature properties of a metric connection with totally skew-symmetric torsion on Riemannian manifold as well as 
the notions and existence of a metric linear connection preserving a given $Spin(7)$--structure and having totally skew-symmetric torsion from \cite{I,FI,IS}. 
\subsection{Metric connection with skew-symmetric  torsion and its curvature}
On a Riemannian manifold $(M,g)$ of dimension $n$ any metric connection $\sb$ with totally skew-symmetric torsion $T$ is connected with the Levi-Civita connection $\sb^g$ of the metric $g$ by
\begin{equation}\label{tsym}
\sb^g=\sb- \frac12T.
\end{equation}
The exterior derivative $dT$ has the following  expression (see e.g. \cite{I,IP2,FI})
\begin{equation}\label{dh}
\begin{split}
dT(X,Y,Z,V)=(\nabla_XT)(Y,Z,V)+(\nabla_YT)(Z,X,V)+(\nabla_ZT)(X,Y,V)\\+2\sigma^T(X,Y,Z,V)-(\nabla_VT)(X,Y,Z),
 \end{split}
 \end{equation}
where the 4-form $\sigma^T$
 is defined by
 \begin{equation}\label{sigma}
 \sigma ^T(X,Y,Z,V)=\frac12\sum_{j=1}^n(e_j\lrcorner T)\wedge(e_j\lrcorner T)(X,Y,Z,V), 
\end{equation} 
$(e_a\lrcorner T)(X,Y)=T(e_a,X,Y)$ is the interior multiplication and $\{e_1,\dots,e_n\}$ is an orthonormal  basis.

The properties of the 4-form $\sigma^T$ are studied in detail in \cite{AFF}, where it is shown that $\sigma^T$ measures the "degeneracy" of the 3-form $T$. 

One  easily gets from \eqref{tsym} \cite{AF}
\begin{equation}\label{dtnt}
\LC T=\sb T+\frac12\sigma^T.
\end{equation}
For the curvature  we use the convention $ R(X,Y)Z=[\nabla_X,\nabla_Y]Z - \nabla_{[X,Y]}Z$ 
 and $ R(X,Y,Z,V)=g(R(X,Y)Z,V)$. It has the well known properties
 \begin{equation}\label{r1}
 R(X,Y,Z,V)=-R(Y,X,Z,V)=-R(X,Y,V,Z).
 \end{equation}
  The first Bianchi identity for $\nabla$ can be written in the  form (see e.g. \cite{I,IP2,FI})
 \begin{equation}\label{1bi}
 \begin{split}
 R(X,Y,Z,V)+ R(Y,Z,X,V)+ R(Z,X,Y,V)\\
 =dT(X,Y,Z,V)-\sigma^T(X,Y,Z,V)+(\nabla_VT)(X,Y,Z).
 \end{split}
 \end{equation}
It is proved in \cite[p. 307]{FI} that the curvature of  a metric connection $\sb$ with totally skew-symmetric torsion $T$  satisfies  the  identity
 \begin{equation}\label{gen}
 \begin{split}
 R(X,Y,Z,V)+ R(Y,Z,X,V)+ R(Z,X,Y,V)-R(V,X,Y,Z)-R(V,Y,Z,X)-R(V,Z,X,Y)\\
 =\frac32dT(X,Y,Z,V)-\sigma^T(X,Y,Z,V).
 \end{split}
 \end{equation}
 We obtain from \eqref{gen} and \eqref{1bi} that  the curvature $R$ satisfies the relation \cite[Proposition~2.1]{IS}
 \begin{equation}\label{1bi1}
 \begin{split}
R(V,X,Y,Z)+R(V,Y,Z,X)+R(V,Z,X,Y)= -\frac12dT(X,Y,Z,V)+(\nabla_VT)(X,Y,Z).
 \end{split}
 \end{equation}
 
\begin{dfn} We say that the curvature $R$ satisfies the Riemannian first Bianchi identity if 
\begin{equation}\label{RB}
R(X,Y,Z,V)+R(Y,Z,X,V)+R(Z,X,Y,V)=0.
\end{equation}
\end{dfn}
 A well known algebraic fact is that \eqref{r1} and \eqref{RB} imply $R\in S^2\Lambda^2$, i.e.
 \begin{equation}\label{r4}
 R(X,Y,Z,V)=R(Z,V,X,Y).
 \end{equation}

 Note that, in general, \eqref{r1} and \eqref{r4} do not imply \eqref{RB}.

It is proved in \cite[Lemma~3.4]{I} that a metric connection $\sb$ with totally skew-symmetric torsion $T$ satifies \eqref{r4}  if and only if the the covariant derivative of the torsion with respect to either to the torsion connection, $\sb T$, or to the Levi-Civita connection, $\LC T$, is a $4$--form, 
\begin{lemma}\cite[Lemma~3.4]{I} The next equivalences hold for a metric connection with torsion 3-form
\begin{equation}\label{4form}
(\sb_XT)(Y,Z,V)=-(\sb_YT)(X,Z,V) \Longleftrightarrow R(X,Y,Z,V)=R(Z,V,X,Y)  \Longleftrightarrow dT=4\LC T.
\end{equation}
\end{lemma}
It was shown  recently that a metric connection $\sb$ with torsion 3-form $T$ satisfies the Riemannian first Bianchi identity exactly when the next identities hold \cite[Theorem~1.2]{IS}
\begin{equation}\label{FBT}
 dT=-2\nabla T=\frac23\sigma^T.
\end{equation}
In this case, the torsion  $T$ is parallel with respect to the  metric connection with torsion 3-form $\frac13T$ \cite{AF}.

 The   Ricci tensors and scalar curvatures of the connections $\LC$ and $\sb$ are related by \cite[Section~2]{FI}  (see also \cite [Prop. 3.18]{GFS})
\begin{equation}\label{rics}
\begin{split}
Ric^g(X,Y)=Ric(X,Y)+\frac12 (\delta T)(X,Y)+\frac14T^2(X,Y),\quad T^2(X,Y)=\sum_{i=1}^ng\big(T(X,e_i),T(Y,e_i)\big),\\
Scal^g=Scal+\frac14||T||^2,\qquad Ric(X,Y)-Ric(Y,X)=-(\delta T)(X,Y),
\end{split}
\end{equation}
where $\delta=(-1)^{np+n+1}*d*$ is the co-differential acting on $p$-forms and $*$ is the Hodge star operator satisfying $*^2=(-1)^{p(n-p)}$.
One  has the general identities for $\alpha\in\Lambda^1$ and $\beta\in \Lambda^k$ 
\begin{equation}\label{star}
\begin{split}
*(\alpha\lrcorner\beta)=(-1)^{k+1}(\alpha\wedge*\beta), \qquad (\alpha\lrcorner\beta)=(-1)^{n(k+1)}*(\alpha\wedge*\beta),\\
*(\alpha\lrcorner*\beta)=(-1)^{n(k+1)+1}(\alpha\wedge\beta),\qquad (\alpha\lrcorner*\beta)=(-1)^{k}*(\alpha\wedge\beta).
\end{split}
\end{equation}
We shall use the next result established in \cite[Theorem~3.8]{IS}
\begin{thrm} \cite[Theorem~3.8]{IS}\label{s2ric}
Let  the curvature $R$ of a Ricci flat metric connection $\sb$ with  torsion 3-form $T$  satisfies $R\in S^2\Lambda^2$, i.e. \eqref{s2l2} holds. Then the norm of the torsion  is  constant, $||T||=const.$
\end{thrm}

\section{$Spin(7)$--structure}

We briefly recall the notion of a $Spin(7)$--structure. Consider
${\mathbb R}^8$ endowed with an orientation and its standard inner
product. Consider the 4-form $\Phi$ on ${\mathbb R}^8$ given by
\begin{eqnarray}\label{s1}
\Phi &=&-e_{0127} +e_{0236} - e_{0347}-e_{0567} +e_{0146} + e_{0245} -
  e_{0135}
 \\ \nonumber&\phantom{=}&
-e_{3456} - e_{1457} - e_{1256}-e_{1234} - e_{2357} -
  e_{1367} +e_{2467},\nonumber
\end{eqnarray}
where $e_{ijkl}$ denotes the 4-form $e_i\wedge e_j\wedge e_k\wedge e_l$.

The 4-form  $\Phi$ is self-dual, $*\Phi=\Phi$, and the 8-form $\Phi\wedge\Phi$ coincides with
 14 times the volume form of ${\mathbb R}^8$. The subgroup of $GL(8,\mathbb R)$ which
fixes $\Phi$ is isomorphic to the double covering $Spin(7)$ of
$SO(7)$ \cite{Br}. Moreover, $Spin(7)$ is a compact
simply-connected Lie group of dimension 21 \cite{Br}. The Lie algebra of $Spin(7)$ is
denoted by $spin(7)$ and it is isomorphic to the 2-forms satisfying  linear equations, namely
 $spin(7)\cong  \{\alpha \in
\Lambda^2(M)|*(\alpha\wedge\ph)=\alpha\}$. We note here the sign difference with \cite{Br}.

The 4-form
$\Phi$ corresponds to a real spinor $\phi$ and therefore,
$Spin(7)$ can be identified as the isotropy group of a non-trivial
real spinor.

We let the expression
$$
 \ph=\frac1{24}\ph_{ijkl}e_{ijkl}
$$
and thus have the  identites (c.f.  \cite{GMW,Kar2})
\begin{eqnarray}\label{p1}\nonumber
\Phi_{ijpq}\Phi_{ijpq} & = & 336;
\\\nonumber
\Phi_{ijpq}\Phi_{ajpq} &=& 42\delta_{ia};
\\\nonumber
\Phi_{ijpq}\Phi_{klpq} &=& 6\delta_{ik}\delta_{jl} -
  6\delta_{il}\delta_{jk} - 4\Phi_{ijkl};\\\label{iden}
\Phi_{ijks}\Phi_{abcs} &=& \delta_{ia} \delta_{jb} \delta_{kc} +\delta_{ib} \delta_{jc} \delta_{ka} +\delta_{ic} \delta_{ja} \delta_{kb}\\\nonumber
&-& \delta_{ia} \delta_{jc} \delta_{kb} -\delta_{ib} \delta_{ja} \delta_{kc} -\delta_{ic} \delta_{jb} \delta_{ka} \\\nonumber
&-& \delta_{ia}\Phi_{jkbc}-\delta_{ja}\Phi_{kibc}-\delta_{ka}\Phi_{ijbc}\\\nonumber
&-& \delta_{ib}\Phi_{jkca}-\delta_{jb}\Phi_{kica}-\delta_{kb}\Phi_{ijca}\\\nonumber
&-& \delta_{ic}\Phi_{jkab}-\delta_{jc}\Phi_{kiab}-\delta_{kc}\Phi_{ijab}.
\end{eqnarray}

A \emph{$Spin(7)$--structure} on an 8-manifold $M$ is by definition
a reduction of the structure group of the tangent bundle to
$Spin(7)$; we shall also say that $M$ is a \emph{$Spin(7)$--manifold}. This can be described geometrically by saying that
there exists a nowhere vanishing global differential 4-form $\Phi$
on $M$ which can be locally written as \eqref{s1}. The 4-form
$\Phi$ is called the \emph{fundamental form} of the $Spin(7)$--manifold $M$ \cite{Bo}.  Alternatively, a $Spin(7)$--structure can be described by the existence of three-fold
vector cross product  on the tangent spaces of $M$ (see e.g. \cite{Gr}).

The fundamental form of a $Spin(7)$--manifold determines a
Riemannian metric  $g$ which 
  is referred  to as the   metric induced by $\Phi$. We write $\LC$ for the associated Levi-Civita
connection and  $||.||^2$ for the tensor norm with respect to $g$. Note the difference by a factor $k!$ with the norm of a k-form (see Convention~\ref{conv}).

In addition, we will freely identify
vectors and co-vectors via the induced metric $g$.

In general, not every  compact 8-dimensional Riemannian spin manifold $M^8$
admits a $Spin(7)$--structure. We explain the precise condition
\cite{LM}. Denote by $p_1(M), p_2(M), {\mathbb X}(M), {\mathbb
X}(S_{\pm})$ the first and the second Pontrjagin classes, the
Euler characteristic of $M$ and the Euler characteristic of the
positive and the negative spinor bundles, respectively. It is well
known \cite{LM} that a compact spin 8-manifold admits a $Spin(7)$--structure if and only if ${\mathbb X}(S_+)=0$ or ${\mathbb X}(S_-)=0$.
The latter conditions are equivalent to $
p_1^2(M)-4p_2(M)+ 8{\mathbb X}(M)=0$, for an appropriate choice
of the orientation \cite{LM}.

Let us recall that a $Spin(7)$--manifold $(M,g,\Phi)$ is said to be
parallel (torsion-free) if the holonomy $Hol(g)$ of the metric $g$ is a subgroup of $Spin(7)$. This is equivalent to saying
that the fundamental form $\Phi$ is parallel with respect to the
Levi-Civita connection of the metric $g$, $\nabla^g\ph=0$.  

M. Fernandez shows in \cite{F} that
$Hol(g)\subset Spin(7)$ if and only if $d\Phi=0$ which is equivalent to $\delta\ph=0$ since $\ph$ is self-dual 4-form  (see also \cite{Br,Sal}).  It was observed  by Bonan that any parallel $Spin(7)$--manifold is Ricci flat
\cite{Bo}. The first known explicit example of complete parallel $Spin(7)$--manifold with $Hol(g)=Spin(7)$ was constructed by Bryant and
Salamon \cite{BS,Gibb}.
The first compact examples of parallel $Spin(7)$--manifolds with
$Hol(g)=Spin(7)$ were constructed by Joyce \cite{J1,J2}.

There are 4 classes of $Spin(7)$--manifolds according to the
Fernandez classification \cite{F} obtained as irreducible $Spin(7)$
representations of  the space $\nabla^g\Phi$.

The Lee form $\theta$ is defined by \cite{C1}
\begin{equation}\label{g2li}
\theta = -\frac{1}{7}*(*d\Phi\wedge\Phi)=\frac17*(\delta\Phi\wedge
\Phi)=\frac1{7}(\delta\ph)\lrcorner\ph,\quad \theta_a=\frac1{42}(\delta\ph)_{ijk}\ph_{ijka},
\end{equation}
where  $\delta=-*d*$ is the co-differential acting on $k$-forms in dimension eight.

The 4 classes of Fernandez classification \cite{F} can be described in
terms of the Lee form as follows \cite{C1}: $W_0 : d\Phi=0; \quad
W_1 : \theta =0; \quad W_2 : d\Phi = \theta\wedge\Phi; \quad W :
W=W_1\oplus W_2.$

A $Spin(7)$--structure of the class $W_1$ (i.e.
$Spin(7)$--structure with zero Lee form) is called
 \emph{a balanced $Spin(7)$--structure}.
If the Lee form is closed, $d\theta=0,$ then the $Spin(7)$--structure is
locally conformally equivalent to a balanced one \cite{I} (see also \cite{Kar1,Kar2}).
It is known due to  \cite{C1} that the Lee form of a $Spin(7)$--structure in the class $W_2$ is closed and therefore such a
manifold is locally conformally equivalent to a parallel $Spin(7)$--manifold. 

If $M$ is compact then it is shown in \cite[Theorem~4.3]{I} that  in every conformal class of  $Spin(7)$--structures $[\ph]$ there exists a unique $Spin(7)$--structure with co-closed Lee form, $\delta\theta=0$. The compact $Spin(7)$--spaces with closed but not exact Lee form
(i.e. the structure is locally but not globally
conformally parallel) have very different topology than the parallel ones
\cite{I,IPP}.

Coeffective cohomology and coeffective numbers of 
a $Spin(7)$--manifold are studied in \cite{Ug}.

\subsection{Decomposition of the space of forms} We take the following description of the decomposition of the space of forms from \cite{Kar2}.

Let $(M, \ph)$ be a $Spin(7)$--manifold. The action of $Spin(7)$  on the tangent space induces an
action of $Spin(7)$ on $\Lambda^k(M)$ splitting the exterior algebra into orthogonal irreducible $Spin(7)$  subspaces, where
$\Lambda^k_l$ corresponds to an $l$-dimensional $Spin(7)$-irreducible subspace of $\Lambda^k$:
\begin{equation*}\label{dec}
\begin{split}
\Lambda^2(M)=\Lambda^2_7\oplus\Lambda^2_{21}, \qquad \Lambda^3(M)=\Lambda^3_8\oplus\Lambda^3_{48},\qquad
\Lambda^4(M)=\Lambda^4_1\oplus\Lambda^4_7\oplus\Lambda^4_{27}\oplus\Lambda^4_{35},
\end{split}
\end{equation*}
where
\begin{equation}\label{dec2}
\begin{split}
\Lambda^2_7=\{\phi\in \Lambda^2(M) | *(\phi\wedge\ph)=-3\phi\};\\
\Lambda^2_{21}=\{\phi\in \Lambda^2(M) | *(\phi\wedge\ph)=\phi\}\cong spin(7);\\
\Lambda^3_8=\{*(\alpha\wedge\ph) |  \alpha\in\Lambda^1\}=\{\alpha\lrcorner\ph\};\\
\Lambda^3_{48}=\{\gamma\in \Lambda^3(M) | \gamma\wedge\ph=0\}.
\end{split}
\end{equation}
Hence, a 2-form $\phi$ decomposes into two $Spin(7)$--invariant parts, $\Lambda^2=\Lambda^2_7\oplus\Lambda^2_{21}$, and
\begin{equation*}
\begin{split}
\phi\in \Lambda^2_7 \Leftrightarrow \phi_{ij}\ph_{ijkl}=-6\phi_{kl},\\
\phi\in \Lambda^2_{21} \Leftrightarrow \phi_{ij}\ph_{ijkl}=2\phi_{kl}.
\end{split}
\end{equation*}
Moreover, following \cite{Kar2} one  considers the operator $D:\Lambda^2\longrightarrow\Lambda^4$ defined for a 2--form $\alpha$ by
\[(D\alpha)_{ijkl}=\alpha_{is}\ph_{sjkl}+\alpha_{js}\ph_{iskl}+\alpha_{ks}\ph_{ijsl}+\alpha_{sl}\ph_{ijks}.
\]
\begin{prop}\cite[Proposition~2.3]{Kar2}\label{spin7}
The kernel of $D$ is isomorphic to $\Lambda^2_{21}\cong spin(7)$.
\end{prop}

For $k>4$ we have $\Lambda^k_l=*\Lambda^{8-k}_l$.

For $k=4$, following \cite{Kar2}, one considers the operator $\Omega_{\ph}:\Lambda^4 \longrightarrow\Lambda^4$ defined as follows
\begin{equation}\label{op}
\begin{split}
(\Omega_{\ph}(\sigma))_{ijkl}=\sigma_{ijpq}\ph_{pqkl}+\sigma_{ikpq}\ph_{pqlj}+\sigma_{ilpq}\ph_{pqjk}+\sigma_{jkpq}\ph_{pqil}+\sigma_{jlpq}\ph_{pqki}+\sigma_{klpq}\ph_{pqij}.
\end{split}
\end{equation}
\begin{prop}\cite[Proposition~2.8]{Kar2}\label{427}
The spaces $\Lambda^4_1,\Lambda^4_7,\Lambda^4_{27},\Lambda^4_{35}$ are all eigenspaces of the operator $\Omega_{\ph}$ with distinct eigenvalues. Specifically,
\begin{equation}\label{dec4}
\begin{split}
\Lambda^4_1=\{\sigma\in\Lambda^4:\Omega_{\ph}(\sigma)=-24\sigma\};\qquad \Lambda^4_7=\{\sigma\in\Lambda^4:\Omega_{\ph}(\sigma)=-12\sigma\};\\\Lambda^4_{27}=\{\sigma\in\Lambda^4:\Omega_{\ph}(\sigma)=4\sigma\}=\{\sigma\in\Lambda^4:\sigma_{ijkl}\ph_{mjkl}=0\};\qquad \Lambda^4_{35}=\{\sigma\in\Lambda^4:\Omega_{\ph}(\sigma)=0\};\\
\Lambda^4_+=\{\sigma\in\Lambda^4:*\sigma=\sigma\}=\Lambda^4_1\oplus\Lambda^4_7\oplus\Lambda^4_{27};\qquad \Lambda^4_-=\{\sigma\in\Lambda^4:*\sigma=-\sigma\}=\Lambda^4_{35}.
\end{split}
\end{equation}
\end{prop}

\section{The $Spin(7)$--connection with skew-symmetric torsion}
The existence of parallel spinors with respect to a metric connection with torsion 3-form  in dimension 8 is important in supersymmetric string theories since the number of parallel spinors determines the number of preserved supersymmetries and this is the first Killing spinor equation in the heterotic Strominger system in dimension eight \cite{GKMW,GMW,GMPW, MS}. 
It is shown in \cite{I} that any $Spin(7)$--manifold $(M,\ph)$ admits a unique $Spin(7)$--connection with totally skew-symmetric torsion. 
\begin{thrm}\cite[Theorem~1]{I}
Let $(M,\ph)$  be a $Spin(7)$--manifold with fundamental 4-form $\ph$. There always exists a unique linear connection $\sb$ preserving the $Spin(7)$--structure, $\sb\ph=\sb g=0,$ with totally skew-symmetric torsion $T$ given by
\begin{equation}\label{torcy}
T=-*d\ph+\frac76*(\theta\wedge\ph)=\delta\ph+\frac76\theta\lrcorner\ph,
\end{equation}
where the Lee form $\theta$ is given by \eqref{g2li}.
\end{thrm}
Note that we use here $\ph:=-\ph$ in \cite{I}.  

See also \cite{Fr,Mer} for subsequent proofs of this theorem.

\subsection{The torsion and the Ricci tensor}
Express the codifferential of a 4-form in terms of the Levi-Civita connection and then in terms of the torsion connection using \eqref{tsym} and $\sb\Phi=0$ to get
\begin{equation}\label{dphi}
\begin{split}
\delta\Phi_{klm}=-\sb^g_j\Phi_{jklm}=-\sb_j\Phi_{jklm}+\frac12T_{jsk}\Phi_{jslm}-\frac12T_{jsl}\Phi_{jskm}+\frac12T_{jsm}\Phi_{jskl}\\
=\frac12T_{jsk}\Phi_{jslm}-\frac12T_{jsl}\Phi_{jskm}+\frac12T_{jsm}\Phi_{jskl}.
\end{split}
\end{equation}
Substitute \eqref{dphi} into \eqref{torcy} to obtain the following expression for the 3-form torsion $T$,
\begin{equation}\label{torcy2}
T_{klm}=\frac12T_{jsk}\Phi_{jslm}-\frac12T_{jsl}\Phi_{jskm}+\frac12T_{jsm}\Phi_{jskl}+\frac76\theta_s\Phi_{sklm}.
\end{equation}
Applying \eqref{iden},  it is straightforward to check from \eqref{g2li} and \eqref{torcy2} that the Lee form $\theta$ can be expressed in terms of the torsion $T$ and  the 4-form $\Phi$ as follows
\begin{equation}\label{tit}
\begin{split}
\theta_i=-\frac17T_{jkl}\Phi_{jkli}.
\end{split}
\end{equation}
Let $d^{\sb}\theta(X,Y)=(\sb_X\theta)Y-(\sb_Y\theta)$ be the skew-symmetric part of $\sb\theta$. We have
\begin{prop}\label{propDT}
On a $Spin(7)$ manifold the following formulas hold true
\begin{equation}\label{deltaT}
\theta\lrcorner\delta\ph=\theta\lrcorner T,\qquad \delta T=\frac76(d\theta\lrcorner\ph-\theta\lrcorner T)=\frac76\Big(d^{\sb}\theta\lrcorner \ph+(\theta\lrcorner T)\lrcorner\ph-\theta\lrcorner T\Big).
\end{equation}
\end{prop}
\begin{proof}
The first formula  follows directly from \eqref{torcy}. 

We get from \eqref{torcy} using \eqref{star}  that
\begin{equation}\label{deltaT1}
\begin{split}
\delta T=-*d*(-*d\ph+\frac76*(\theta\wedge\ph))=\frac76*(d\theta\wedge\ph-\theta\wedge d\ph)=\frac76(d\theta\lrcorner\ph+\theta\lrcorner *d*\ph)\\
=\frac76(d\theta\lrcorner\ph-\theta\lrcorner\delta\ph)=\frac76(d\theta\lrcorner\ph-\theta\lrcorner T),
\end{split}
\end{equation}
where we applied the already established first formula in \eqref{deltaT} to achieve the last equality.

The equality \eqref{tsym} yields
\begin{equation}\label{nth}
\LC\theta=\sb\theta+\frac12\theta\lrcorner T,\qquad d\theta=d^{\sb}\theta+\theta\lrcorner T.
\end{equation}
Substitute \eqref{nth} into \eqref{deltaT1} to obtain the third identity in \eqref{deltaT}
\end{proof}

\subsection{The Ricci tensor of the torsion connection}
The Ricci tensor $Ric$ and the scalar curvature $Scal$ of the torsion connection were calculated in \cite{I1} with the help of the properties of the $\sb$-parallel real spinor corresponding to the $Spin(7)$ form $\ph$, applying the Schr\"odinger-Lichnerowicz formula for the torsion connection, established in \cite{I1}. Here we calculate the Ricci tensor and its scalar curvature directly to make the paper more self-contained.  We have
\begin{thrm}\cite{I1}
The Ricci tensor and the scalar curvature  of the $Spin(7)$-torsion connection are given by
\begin{equation}\label{ricg2}
\begin{split}
Ric_{ij}=-\frac1{12}dT_{iabc}\ph_{jabc}-\frac76\sb_i\theta_j;\qquad
Scal=\frac72\delta\theta+\frac{49}{18}||\theta||^2-\frac13||T||^2.
\end{split}
\end{equation}
The Riemannian scalar curvature $Scal^g$ of a $Spin(7)$--manifold has the expression
\begin{equation}\label{scal1}
\begin{split}
Scal^g=\frac72\delta\theta+\frac{49}{18}||\theta||^2-\frac1{12}||T||^2.
\end{split}
\end{equation}
\end{thrm}
\begin{proof}
 Since $\sb\ph=0$, the curvature $R$ of the $Spin(7)$ torsion connection lies in the Lie algebra $spin(7)$,  
\begin{equation}\label{rr}
\begin{split}
R(X,Y,e_i,e_j)\ph(e_i,e_j,Z,V)=2R(X,Y,Z,V),\qquad R_{ijab}\ph_{abkl}=2R_{ijkl}.
\end{split}
\end{equation}
We have from \eqref{rr} using \eqref{1bi1}, \eqref{tit} and \eqref{dh}  that the Ricci tensor $Ric$ of  $\sb$ is given by
\begin{multline}\label{ricdt}
2Ric_{ij}=-R_{iabc}\ph_{jabc}=-\frac13\Big(R_{iabc}+R_{ibca}+R_{icab} \Big)\ph_{jabc}=-\frac16dT_{iabc}\ph_{jabc}-\frac13\sb_iT_{abc}\ph_{jabc}\\=-\frac16dT_{iabc}\ph_{jabc}-\frac73\sb_i\theta_j,
\end{multline}
which completes the proof of the first identity in \eqref{ricg2}.

We calculate  from \eqref{torcy2} using \eqref{iden} that
\begin{equation}\label{ng4}
\begin{split}
\sigma^T_{jabc}\ph_{jabc}=3T_{jas}T_{bcs}\ph_{jabc}=2||T||^2-\frac{49}3||\theta||^2.
\end{split}
\end{equation}
We get from \eqref{dh}, applying \eqref{tit} and  \eqref{ng4} that 
\begin{equation}\label{g22}
\begin{split}
dT_{jabc}\ph_{jabc}=4\sb_jT_{abc}\ph_{jabc}+2\sigma^T_{jabc}\ph_{jabc}=-28\delta\theta+4||T||^2-\frac{98}3||\theta||^2.
\end{split}
\end{equation}
Take the trace in the first identity in \eqref{ricg2},  substitute  \eqref{g22} into the obtained equality 
to get the second identity in \eqref{ricg2}. The equality  \eqref{scal1} follows from \eqref{rics} and  the second identity in \eqref{ricg2}.
\end{proof}
\begin{rmrk}
For the $\Lambda^3_{48}$ component $(\delta\ph)^3_{48}$ of $\delta\ph$ we get  using  \eqref{dphi}, \eqref{iden} and \eqref{tit} that
\begin{equation*}
(\delta\ph)^3_{48}=\delta\ph+\theta\lrcorner\ph
\end{equation*}
which combined with 
\eqref{torcy} yields the next expressions  for  $T$ and its norm $||T||^2$ in terms of $(\delta\ph)^3_{48}$,
\begin{equation}\label{torcy3}
\begin{split}
T=(\delta\ph)^3_{48}+\frac16\theta\lrcorner\ph, 
\qquad ||T||^2=||(\delta\ph)^3_{48}||^2+\frac{7}6||\theta||^2.
\end{split}
\end{equation}
A substitution of \eqref{torcy3} into \eqref{ricg2}, \eqref{scal1}, \eqref{ng4} and \eqref{g22} gives  expressions of these formulas in terms of $(\delta\ph)^3_{48}$.
\end{rmrk}
As a consequence of \eqref{ricg2}, we get  the  result, first established by Bonan \cite{Bo} for  parallel $Spin(7)$--spaces,
\begin{cor}\cite{Bo}\label{ricB}
If the curvature of the $Spin(7)$-torsion connection  satisfies the Riemannian first Bianchi identity then its Ricci tensor vanishes.
\end{cor}
\begin{cor}\label{thconj2}
Let $(M,\ph)$ be a  balanced $Spin(7)$--manifold, $\theta=0$. Then any one of the following three conditions, $dT=0$, $Scal^g=0$, $Scal=0$ imply that $(M,\ph)$  is parallel, $\LC\ph=0$.
\end{cor}
\begin{proof} The conclusions of the corollary follow from \eqref{ricg2}, \eqref{scal1} and \eqref{g22}.
\end{proof}
\begin{thrm}\label{mainsu1}
Let $(M,\ph)$ be a $Spin(7)$--manifold. The Ricci tensor of the $Spin(7)$-torsion connection is symmetric 
 if and only if the 2-form
$d^{\sb}\theta$ 
 is given by 
\begin{equation}\label{new}
3d^{\sb}\theta=-\theta\lrcorner T+(\theta\lrcorner T)\lrcorner\ph=-\theta\lrcorner \delta\ph+(\theta\lrcorner\delta\ph)\lrcorner\ph =-d^{\sb}\theta\lrcorner\ph.
\end{equation}
In particular, $d^{\sb}\theta$ belongs to $\Lambda^2_7$.
\end{thrm}
\begin{proof}
The Ricci tensor of $\sb$ is symmetric exactly when $\delta T=0$ by \eqref{rics}. 
The second equality of \eqref{deltaT} shows that $\delta T=0$ if and only if
\begin{equation}\label{new1}
d^{\sb}\theta_{st}\ph_{stlm}=2\theta_kT_{klm}-\theta_pT_{pst}\ph_{stlm},
\end{equation}
which multiplied by $\ph_{lmab}$  yields, using \eqref{iden},
\begin{equation}\label{the2}
\begin{split}
-4d^{\sb}\theta_{st}\ph_{stab}+12d^{\sb}\theta_{ab}=2\theta_kT_{klm}\ph_{lmab}+4\theta_pT_{pst}\ph_{stab}-12\theta_pT_{pab}.
\end{split}
\end{equation}
Apply  \eqref{new1} to \eqref{the2} to obtain \eqref{new}.
\end{proof}
On a locally conformally parallel $Spin(7)$--manifold we have $d\ph=\theta\wedge\ph$ and  \eqref{torcy} reads $T=\frac16*d\ph$, which yields $\delta T=0$ and the Ricci tensor of the $Spin(7)$--torsion connection of a locally conformally parallel $Spin(7)$--manifold is symmetric due to \eqref{rics}.
The structure of compact locally conformally parallel $Spin(7)$--manifolds is described in \cite{IPP}.

If $d\theta=0=\delta T$ then  \eqref{deltaT} and \eqref{nth} yield the following
\begin{cor}Let $(M,\ph)$ be a $Spin(7)$--manifold with symmetric Ricci tensor of the $Spin(7)$--torsion connection and closed Lee form. Then 
$
 \theta\lrcorner T=\theta\lrcorner\delta\ph=0$ and $\sb\theta=\LC\theta.
 $
\end{cor}
More precisely, we have 
\begin{prop}\label{maincor}
Let $(M,\ph)$ be a $Spin(7)$--manifold with  symmetric Ricci tensor of the $Spin(7)$--torsion connection.
\hspace{0.1cm}The following three conditions are equivalent:
\begin{itemize}
\item[a)] The covariant derivative of the Lee form $\theta$ with respect to $\sb$ is symmetric.
\item[b)] The 2-form $\theta\lrcorner \delta\ph=\theta\lrcorner T$ belongs to $\Lambda^2_{21}\cong spin(7)$.
\item[c)] The 2-form $d\theta$ belongs to $\Lambda^2_{21}\cong spin(7)$.
\end{itemize}
\end{prop}
\begin{proof} The equivalence of a) and b) follows from \eqref{new}. 

Since $\delta T=0$, we get from \eqref{deltaT1}  $d\theta\lrcorner\ph=\theta\lrcorner T$, which proves the equivalence of b) and c).
\end{proof}
\section{Proof of Theorem~\ref{mainsu3}}
To prove  Theorem~\ref{mainsu3}, we start  with the next 
\begin{lemma}\label{sbtheta}
Let $(M,\ph)$ be a $Spin(7)$--manifold  with $d\theta\in spin(7)$.

If the $Spin(7)$--torsion connection $\sb$  is Ricci-flat and  has curvature $R \in S^2\Lambda^2$, i.e. \eqref{s2l2} holds, 
 then 
 $$\sb\theta=0.$$
In particular, the Lee form is co-closed, $\delta\theta=0$.
\end{lemma}
\begin{proof}
The condition $R\in S^2\Lambda^2$ is equivalent to $\sb T$ to be a 4-form because of \eqref{4form}. 
Substitute \eqref{dh} into \eqref{ricg2} to get using \eqref{tit} that
\begin{equation}\label{su1}
0=Ric_{ij}+\frac76\sb_i\theta_j+\frac1{12}(4\sb_iT_{abc}+2\sigma^T_{iabc})\ph_{jabc}=Ric_{ij}+\frac72\sb_i\theta_j+\frac1{6}\sigma^T_{iabc}\ph_{jabc}.
\end{equation}
Let $Ric=0$. Then we have from \eqref{tit}, \eqref{su1}, \eqref{ricg2} and \eqref{sigma}
\begin{equation}\label{li}
\begin{split}
\sb_i\theta_j=\frac17\sb_iT_{abc}\ph_{jabc}=-\frac1{12}dT_{iabc}\ph_{jabc}=-\frac1{21}\sigma^T_{iabc}\ph_{jabc}\\
=-\frac1{21}\Big(T_{abs}T_{sci}+T_{bcs}T_{sai}+T_{cas}T_{sbi}\Big)\ph_{abcj}=-\frac17T_{abs}T_{cis}\ph_{abcj}.
\end{split}
\end{equation}
We calculate from \eqref{li} using \eqref{torcy2}
\begin{equation}\label{part}
\begin{split}
-7\sb_p\theta_k=T_{jsl}T_{lmp}\ph_{jsmk}
=T_{klm}T_{lmp}-\frac12T_{jsk}\ph_{jslm}T_{lmp}-\frac76\theta_a\ph_{aklm}T_{lmp}.
\end{split}
\end{equation}
Multiply \eqref{part} with $\theta_p$ and apply a) and b) of Proposition~\ref{maincor} 
to get
\begin{equation}\label{part1}
\begin{split}
-\frac72\sb_k||\theta||^2=-7\theta_p\sb_k\theta_p=-7\theta_p\sb_p\theta_k=\Big(T_{klm}T_{lmp}-\frac12T_{jsk}\ph_{jslm}T_{lmp}-\frac76\theta_a\ph_{aklm}T_{lmp}\Big)\theta_p=0.
\end{split}
\end{equation}
Since $Scal=0$, the second identity in \eqref{ricg2} yields
\begin{equation}\label{delth}
\delta\theta=-\frac79||\theta||^2+\frac2{21}||T||^2.
\end{equation}
The condition \eqref{s2l2} together with Theorem~\ref{s2ric} tells us  that the norm of the torsion is constant.
The norm of the Lee form $\theta$ is also a  constant due to \eqref{part1}. Now, \eqref{delth}  shows that $\delta\theta$ is a constant, 
\begin{equation}\label{delc}
\sb_k\delta\theta=-\sb_k\sb_i\theta_i=0.
\end{equation}
Using \eqref{part1}, $d^{\sb}\theta=0$, 
 \eqref{delc}  and the Ricci identity for the torsion connection $\sb$, we have the   equalities
\begin{multline*}
0=\frac12\sb_i\sb_i||\theta||^2=\theta_j\sb_i\sb_i\theta_j+||\sb\theta||^2=\theta_j\sb_i\sb_j\theta_i+||\sb\theta||^2\\=\theta_j\sb_j\sb_i\theta_i
-R_{ijis}\theta_s\theta_j-\theta_jT_{ijs}\sb_s\theta_i+||\sb\theta||^2=Ric_{js}\theta_j\theta_s+||\sb\theta||^2=||\sb\theta||^2,
\end{multline*}
since $Ric=0$. 
The proof of the Lemma is completed due to the equality $\LC_i\theta_i=\sb_i\theta_i+\frac12\theta_sT_{sii}=\sb_i\theta_i$, which is a consequence of  \eqref{tsym}.
\end{proof}
In view of \eqref{nth} and Lemma~\ref{sbtheta} we derive
\begin{cor}\label{sbthetad}
Let $(M,\ph)$ be a $Spin(7)$--manifold with closed Lee form, $d\theta=0$. 
If the  $Spin(7)$--torsion  connection $\sb$  is Ricci-flat and  has curvature $R \in S^2\Lambda^2,$ i.e. \eqref{s2l2} holds,  
 then the Lee form $\theta$ is $\sb$-parallel and $\LC$-parallel,   $\sb\theta=\LC\theta=0$.
\end{cor}
To finish the proof of Theorem~\ref{mainsu3}, since $R\in S^2\Lambda^2$ and $Ric=0$, we observe  from \eqref{li}, \eqref{tit}, \eqref{ricg2}  and Lemma~\ref{sbtheta} the validity of the following  identities
\begin{equation}\label{nthh}
\begin{split}
\sb_pT_{jkl}\Phi_{jkli}=7\sb_p\theta_i=0;\\
\sigma^T_{pjkl}\Phi_{jkli}=-21\sb_p\theta_i=0;\\
dT_{pjkl}\ph_{jkli}=-14\sb_p\theta_i=0.
\end{split}
\end{equation}
The identities \eqref{nthh}  show that the 4-forms $(\sb T)\in \Lambda^4_{27}$, $(\sigma^T)\in \Lambda^4_{27}$ and $dT\in\Lambda^4_{27}$.

In particular, the 4-forms $\sb T,\sigma^T$ and $dT$ are self-dual due to  Proposition~\ref{427}. Hence, we have
\begin{equation}\label{dtt}
\delta dT=-*d*dT=-*d^2T=0.
\end{equation}
If $M$ is compact,  take the integral scalar product with $T$ and use \eqref{dtt} to get
\[0=<\delta dT,T>=\int_M|dT|^2 vol.
\]
Hence, $dT=0$ and \eqref{4form} yields $0=dT=4\LC T$.
 
 The converse follows from the next
 \begin{lemma}\label{lnew}
 Let $(M,\Phi)$ be a $Spin(7)$--manifold with $\LC T=0$. Then the $Spin(7)$--torsion connection $\sb$ is Ricci flat, has curvature $R\in S^2\Lambda^2$ and  $\sb\theta=0$. In particular, $\delta\theta=0$.
 \end{lemma}
 \begin{proof}
If $\LC T=0$ then $dT=\delta T=0$. Consequently,  $R\in S^2\Lambda^2$ because of \eqref{4form}, the Ricci tensor  is symmetric due to \eqref{rics} and $Ric=-\frac76\sb\theta$  by \eqref{ricg2}. In particular  $d^{\sb}\theta=0$.  

Now, we apply Proposition~\ref{maincor} to  conclude that each of the two forms $\theta\lrcorner T$, $d\theta$ belongs to $ \Lambda^2_{21}$. 

We will prove that $\sb\theta=0$ following the steps in the proof of Lemma~\ref{sbtheta}. First we show $d||\theta||^2=0$.

 The condition $dT=0$ and  \eqref{dtnt} imply $\sb T=-\frac12\sigma^T$ which combined with \eqref{tit}  and  \eqref{torcy2} yelds
\begin{equation}\label{liii}
\begin{split}
14\sb_p\theta_k=2\sb_pT_{jsm}\ph_{kjsm}=-\sigma^T_{pjsm}\ph_{kjsm}=-3T_{jsl}T_{lmp}\ph_{jsmk}\\
=-3(T_{klm}T_{lmp}-\frac12T_{jsk}\ph_{jslm}T_{lmp}-\frac76\theta_a\ph_{aklm}T_{lmp}).
\end{split}
\end{equation}
Multiply \eqref{liii} with $\theta_p$ and use b) of Proposition~\ref{maincor} 
to obtain
\begin{equation}\label{part11}
\begin{split}
7\sb_k||\theta||^2=14\theta_p\sb_k\theta_p=14\theta_p\sb_p\theta_k=-3(T_{klm}T_{lmp}-\frac12T_{jsk}\ph_{jslm}T_{lmp}-\frac76\theta_a\ph_{aklm}T_{lmp})\theta_p=0.
\end{split}
\end{equation}
The identity \eqref{g22} and $dT=0$ yield
\begin{equation}\label{delth1}
\delta\theta=\frac17||T||^2-\frac76||\theta||^2, \qquad \sb\delta\theta=0.
\end{equation} 
The second equality in \eqref{delth1} is a consequence of the first one and the facts that $\sb||T||^2=\sb||\theta||^2=0$.

 Using \eqref{part11}, $d^{\sb}\theta=0$, the second equality in \eqref{delth1}  and the Ricci identity for the torsion connection $\sb$, we have the  sequence of equalities
\begin{multline*}
0=\frac12\sb_i\sb_i||\theta||^2=\theta_j\sb_i\sb_i\theta_j+||\sb\theta||^2=\theta_j\sb_i\sb_j\theta_i+||\sb\theta||^2=\theta_j\sb_j\sb_i\theta_i
-R_{ijis}\theta_s\theta_j-\theta_jT_{ijs}\sb_s\theta_i+||\sb\theta||^2\\=Ric_{js}\theta_j\theta_s+||\sb\theta||^2=-\frac76\theta_j\sb_j\theta_s\theta_s+||\sb\theta||^2=-\frac7{12}\theta_j\sb_j||\theta||^2+||\sb\theta||^2=||\sb\theta||^2,
\end{multline*}
since $Ric=-\frac76\sb\theta$, $\sb\theta$ is symmetric and $||\theta||^2=const$. Hence, $Ric=0$. 
 \end{proof}
 Thus, the proof of Theorem~\ref{mainsu3} is completed.
 \subsection{Proof of  Theorem~\ref{co1}, Corollary~\ref{mainspin} and Corollary~\ref{co2}}
To proof Theorem~\ref{co1} we recall that the Riemannian first Bianchi identity \eqref{RB} for the torsion connection $\sb$ implies \eqref{r4} and the vanishing of its Ricci tensor (cf. Corollary~\ref{ricB}). Hence, \eqref{s2l2} holds true and Theorem~\ref{mainsu3} shows $\LC T=dT=0$. On the other hand, \eqref{RB} is equivalent to the conditions \eqref{FBT} (cf.  \cite[Theorem~1.2]{IS}), which combined with $dT=0$, completes the proof of Theorem~\ref{co1}.

Finally,  the proofs of Corollary~\ref{mainspin} and Corollary~\ref{co2} follow from the proof of Theorem~\ref{mainsu3}, Theorem~\ref{co1}  and Corollary~\ref{sbthetad}.

\section{$Spin(7)$--manifolds with closed torsion }\label{closed}
In this section we describe properties of compact $Spin(7)$--manifold with closed torsion form, proof our main results  and show that these spaces are generalized gradient Ricci solitons.
\subsection{Proof of Theorem~\ref{closT}}
\begin{proof}
The condition $dT=0$ and  the equality \eqref{ricg2} imply \eqref{clos1}.

For the converse, assume \eqref{clos1} holds.  The equality   \eqref{clos1} combined with the first identity in  \eqref{ricg2} yields 
$dT_{jabc}\Phi_{iabc}=0$, which shows $dT\in \Lambda^4_{27}$ and in particular $dT$ is self-dual, due to  Proposition~\ref{427}. Hence $\d dT=0$  implying $dT=0,$ since $M$ is compact.

Further, a combination  of \eqref{clos1} with \eqref{rics} and \eqref{nth} yields 
$
\frac67\delta T=d^{\nabla}\theta=d\theta-\theta\lrcorner T
$ 
which compared  with \eqref{deltaT1} implies $d\theta=d\theta\lrcorner\Phi$.
Hence $d\theta\in\Lambda^2_{21}$ and the proof is completed.
\end{proof}
Note that Theorem~\ref{closT} 
 generalizes \cite[Proposition~5.8]{Wit}.
\begin{cor}\label{cclosT}
Let $(M,\ph)$ be a  $Spin(7)$--manifold with closed torsion 3-form, $dT=0$. 
Then 
\eqref{clos1} holds true,   the  exterior derivative of the Lee form $ d\theta\in \Lambda^2_{21}\cong spin(7)$ and $\delta T=\frac76d^{\sb}\theta$.
\end{cor}
\subsection{Proof of Theorem~\ref{closTt}}
\begin{proof}
The equivalences of a) and f)  as well as  between c) and e) follow from Theorem~\ref{closT}.

The second Bianchi identity for the torsion connection reads \cite[Proposition~3.5]{IS}
\begin{equation}\label{e1}
d(Scal)_j-2\sb_iRic_{ji}+\frac16d||T||^2_j+\delta T_{ab}T_{abj}+\frac16T_{abc}dT_{jabc}=0.
\end{equation}
From Theorem~\ref{closT}  we have taking into account \eqref{rics} and \eqref{nth}
\begin{equation}\label{nnewt}Ric=-\frac76\sb\theta, \quad Scal =\frac76\delta \theta, \quad \delta T=\frac76d^{\sb}\theta=\frac76(d\theta-\theta\lrcorner T).
\end{equation}
Since $dT=0$,  the covariant derivative of the first equation in \eqref{delth1} yields
\begin{equation}\label{tt1}
-\sb_j\delta \theta-\frac73\sb_j\theta_s.\theta_s+\frac17\sb_j||T||^2=0.
\end{equation}
Another covariant derivative of \eqref{tt1} together with \eqref{nnewt} implies 
\begin{equation}\label{tt2}
\Delta\delta\theta-\frac73\sb_j\sb_j\theta_s.\theta_s-\frac{12}7||Ric||^2-\frac17\Delta||T||^2=0,
\end{equation}
where $\Delta$ is the Laplace operator $\Delta=-\LC_i\LC_if=-\sb_i\sb_if$ acting on smooth function $f$ since the torsion of $\sb$ is  a 3-form.

On the other hand,  we obtain using \eqref{nnewt} that
\begin{equation}\label{is1}
\sb_i\sb_j\theta_i=\sb_i(\sb_i\theta_j-\frac67\delta T_{ij})=\sb_i\sb_i\theta_j
-\frac67\sb_i\delta T_{ij}=\sb_i\sb_i\theta_j-\frac37\delta T_{ia}T_{iaj},
\end{equation}
where we used the next identity  for a metric connection with skew torsion shown in \cite[Proposition~3.2]{IS}
\begin{equation}\label{iii}
\sb_i\delta T_{ij}=\frac12\delta T_{ia}T_{iaj}.
\end{equation}
We include  a proof  of \eqref{iii} for completeness.  The identity $\delta^2=0$ together with \eqref{tsym} and \eqref{dtnt} imply 
$$0=\delta^2 T_k=\LC_i\LC_jT_{ijk}=\LC_i\sb_jT_{ijk}=\sb_i\delta T_{ik}+\frac12T_{iks}\delta T_{is}=\sb_i\delta T_{ik}-\frac12\delta T_{is}T_{isk}.$$
The equalities \eqref{is1}, \eqref{e1}  and \eqref{nnewt} yield
\begin{equation}\label{is2}
\sb_j\delta \theta+2\sb_i\sb_i\theta_j+\frac17\sb_j||T||^2=0.
\end{equation}
Substitute the second term in \eqref{is2} into \eqref{tt2} to get
\begin{equation}\label{fmax}
\Delta\Big(\delta\theta-\frac17||T||^2 \Big)+\theta_j\sb_j\Big( \frac76\delta\theta+\frac16||T||^2\Big)=\frac{12}7||Ric||^2\ge 0.
\end{equation}
Assume the condition a). Then \eqref{clos1} 
 implies $\delta\theta=0$. Since $M$ is compact  we may apply to \eqref{fmax}  the strong maximum principle (see e.g. \cite{YB,GFS})  to achieve $||T||^2=const.$ Hence, b) follows. 

Suppose b) holds. Since $M$ is compact and $d||T||^2=0$,  the strong maximum principle applied to \eqref{fmax}  implies $\delta\theta=const.=0$ Now \eqref{fmax} shows $Ric$=0. Thus, a) is equivalent to b).

Further, the condition $\delta\theta=0$ implies $d||T||^2=0$ by the strong maximum principle applied to \eqref{fmax}. Hence b) is equivalent to c)

 Finally, to show the equivalence of d) and b), we use \eqref{rics} and \eqref{nnewt} to  write \eqref{fmax} in the form
\begin{equation}\label{fmaxf}
\Delta\Big(\frac67Scal^g-\frac5{14}||T||^2 \Big)+\theta_j\sb_j\Big( Scal^g-\frac1{12}||T||^2\Big)=\frac{12}7||Ric||^2\ge 0.
\end{equation}
The strong maximum principle applied to \eqref{fmaxf} implis d) is equivalent to b) 
which completes
 the proof of Theorem~\ref{closTt}.
\end{proof}
\subsection{Generalized steady Ricci solitons}
It is shown in \cite[Proposition~4.28]{GFS} that a Riemannian manifold $(M,g,T)$ with a closed torsion $dT=0$ is a steady generalized  Ricci soliton if there exists a vector field $X$ and a 2-form $B$ such that it is a solution to the equations
\begin{equation}\label{gein3}
Ric^g=\frac14T^2-\frac12\mathbb{L}_Xg, \qquad \delta T=B,
\end{equation}
for $ B$ satisfying $d(B+X\lrcorner T)=0.$
In particular $\Delta_dT=\mathbb{L}_XT$ where  $\Delta_d=-(d\delta+\delta d)$ is the Hodge Laplacian.

We also  recall  the equivalent formulation \cite[Definition~4.31]{GFS} that a compact Riemannian manifold $(M,g,T)$ with a closed 3-form $T$ is a steady generalized  Ricci soliton with $k=0$ if there exists a vector field $X$ and one has 
\begin{equation}\label{gein2}
Ric^g=\frac14T^2-\frac12\mathbb{L}_Xg, \qquad \delta T=-X\lrcorner T, \qquad dT=0.
\end{equation}
If the vector field $X$ is a gradient of a smooth function $f$ then one has the notion of a steady generalized gradien Ricci soliton.
\begin{prop}\label{grsol}
Let $(M,\Phi)$ be a compact $Spin(7)$--manifold with closed torsion form, $dT=0$. 

Then it is a steady generalized  Ricci soliton with $X=\frac76\theta$ and $B=\frac76(d\theta-\theta\lrcorner T)=\frac76d^{\sb}\theta$.
\end{prop}
\begin{proof}
As we identify the vector field $X$ with its corresponding 1-form via the metric, we have using \eqref{tsym}
\begin{equation}\label{acy1}
dX=d^{\sb}X+X\lrcorner T.
\end{equation}
In view of \eqref{rics}, \eqref{tsym} and \eqref{acy1} we write the first equation in \eqref{gein3} in the form
\begin{equation}\label{acy2}
Ric =-\frac12\delta T-\frac12\mathbb{L}_Xg=-\frac12\delta T-\sb X+\frac12d^{\sb}X=-\frac12\delta T-\sb X+\frac12dX-\frac12X\lrcorner T.
\end{equation}
Set $X=\frac76\theta, \quad B=\frac76(d\theta-\theta\lrcorner T)$ and apply Corollary~\ref{cclosT} to conclude that \eqref{acy2} is trivially satisfied and $\d T=B$. Hence, $(M,\Phi)$ is a steady generalized  Ricci soliton since $d(B+\frac76\theta\lrcorner T)=\frac76d^2\theta=0$.
\end{proof}
One fundamental consequence of Perelman's energy formula for Ricci flow is that compact steady solitons for Ricci flow are automatically gradient.  Adapting these energy functionals to generalized Ricci
flow, it is proved in  \cite[Chapter~6]{GFS}  that steady generalized Ricci solitons on compact manifolds are automatically gradient, and moreover satisfy k = 0, i.e. there exists a smooth function $f$ such that $X=grad(f)$ and \eqref{gein2} takes the form
\begin{equation}\label{gein1}
Ric^g=\frac14T^2-(\LC)^2f, \qquad \delta T=-df\lrcorner T, \qquad dT=0.
\end{equation}
The smooth function $f$ is determined with $u=\exp(-\frac12f)$, where $u$ is  the first eigenfunction of the
Schr\"odinger operator, (see \cite[Lemma~6.3, Corollary~6.10, Corollary~6.11]{GFS}),
\begin{equation}\label{schro}-4\Delta_d +Scal^g-\frac1{12}||T||^2=-4\Delta_d +Scal+\frac16||T||^2.
\end{equation}
In terms of the torsion connection  the steady  generalized gradient Ricci soliton condition \eqref{gein1} can be written in the form (see \cite{IS})
\begin{equation}\label{gein4}
Ric=-\sb^2f, \qquad \d T=-df\lrcorner T, \qquad dT=0.
\end{equation}
\begin{thrm}\label{inf}
Let $(M,\ph)$ be a compact $Spin(7)$--manifold with closed torsion, $dT=0$. The next two conditions are equivalent.
 \begin{itemize}
\item[a)] $(M,\ph)$  is a steady generalized gradient Ricci soliton, i.e. there exists a smooth function $f$ determined with $u=\exp(-\frac12f)$, where $u$ is  the first eigenfunction of the Schr\"odinger operator \eqref{schro}, such that \eqref{gein1}, equivalently \eqref{gein4},  hold. 
\item[b)]  For $f$ determined  by the first eigenfunction $u$ of the  Schr\"odinger operator \eqref{schro} with  $u=\exp(-\frac12f)$, the vector field $$V=\frac76\theta-df \quad  is \quad \sb-parallel, \quad \sb V=0.$$
\end{itemize}
The $\sb$-parallel vector field $V$  determines $d\theta$ and  preserves the $Spin(7)$--structure $(g,\ph)$,
\begin{equation}\label{vsu3}
\frac76d\theta=V\lrcorner T, \quad \mathbb{L}_Vg=\mathbb{L}_V\ph=0.
\end{equation}
If $V$ vanishes at one point then $V=0$. In this case $T=0$, the compact $Spin(7)$--manifold is parallel, $\LC\ph=0$, and $f=const$.
\end{thrm}
\begin{proof}
The first statement a) follows  from Proposition~\ref{grsol} and the general considerations in \cite[Chapter~6]{GFS},  \cite[Lemma~6.3, Corollary~6.10, Corollary~6.11]{GFS}.

To prove b) follows from a), observe that the  $dT=0$ implies \eqref{nnewt}  which 
combined with \eqref{gein4} yield \[\sb(\frac76\theta-df)=0,\qquad \frac76d\theta=V\lrcorner T.\]
For the converse, the condition b)  in the dorm $\frac76\sb\theta=\sb df$ combined with  Theorem~\ref{closT}  and \eqref{rics},  gives
\[Ric_{ij}=-\frac76\sb_i\theta_j=-\sb_i\sb_jf,\quad -\d T_{ij}=Ric_{ij}-Ric_{ji}=df_sT_{sij},\]
since $0=d^2f=\LC_i\LC_jf-\LC_j\LC_if=\sb_i\sb_jf-\sb_j\sb_if+df_sT_{sij}$. 
Hence, \eqref{gein4} holds, which proves the equivalence between a) and  b).

Further, we have $( \mathbb{L}_Vg)_{ij}=\LC_iV_j+\LC_jV_i=\sb_iV_j+\sb_jV_i=0$, since $\sb V=0$. Hence $V$ is Killing.

We calculate the Lie derivative of $\ph$ 
(see e.g. \cite{KN}), 
using  $\sb V=0$. We have 
\begin{multline*}
( \mathbb{L}_V\ph)(X,Y,Z,U)=V\ph(X,Y,Z,U)-\ph([V,X],Y,Z,U)-\ph(X,[V,Y],Z,U)-\ph(X,Y,[V,Z],U)\\-\ph(X,Y,Z,[V,U])=V\ph(X,Y,Z,U)-\ph(\sb_VX,Y,Z,U)-\ph(X,\sb_VY,Z,U)-\ph(X,Y,\sb_VZ,U)\\+T(V,X,e_a)\ph(e_a,Y,Z,U)+T(V,Y,e_a)\ph(X,e_a,Z,U)\\+T(V,Z,e_a)\ph(X,Y,e_a,U)+T(V,U,e_a)\ph(X,Y,Z,e_a).
\end{multline*}
In view of  the already proved first equality in \eqref{vsu3}, we obtain from the last identity that
\begin{equation*}
(\mathbb{L}_V\ph)_{ijkl}=(\sb_V\ph)_{ijkl}+\frac76\Big(d\theta_{is}\ph_{sjkl}+d\theta_{js}\ph_{iskl}+d\theta_{ks}\ph_{ijsl}+d\theta_{ls}\ph_{ijks}  \Big)=0,
\end{equation*}
where we use $\sb \ph=0$, the fact that $d\theta\in spin(7)\cong\Lambda^2_{21}$, due to Corollary~\ref{cclosT}, and apply Proposition~\ref{spin7}  to achieve the last identity. This completes the proof of \eqref{vsu3}.

If $V=0$ then the Lee form is an exact form, $\theta=\frac67df$. This combined with $dT=0$ implies $T=df=0$ since M is compact, fact  well known in physics (see \cite{GMW,GMPW}). 

For completeness, we give  a  different proof relying on the equality  \eqref{g22}. 

Insert $\theta=\frac67df$ into \eqref{g22} and use $dT=0$ to get
\begin{equation}\label{exac}
||T||^2=6(\Delta f+||df||^2)=-6e^{f}\Delta u,
\end{equation}
where $\Delta f$ is the Laplace operator $\Delta f=-\LC_i\LC_if=-\sb_i\sb_if$ acting on smooth function $f$, since the torsion of $\sb$ is  a 3-form and the smooth function  $u=e^{-f}$.

Multiply  \eqref{exac} with $e^{-f}$  and integrate the obtained equality on the compact $M$ to achieve $T=0$. Therefore, $df=0$ and $\LC\ph=0$.

\end{proof}

\subsection{ Examples.}\label{bi} Compact $Spin(7)$--structures with closed torsion became a part of the generalized $Spin(7)$--structures  considered first by F. Witt in \cite{Wit}. A generalized  integrable $Spin(7)$--structure  consists of two   $Spin(7)$--structures $(\ph,\tilde{\ph})$ generating the same Riemannian metric with opposite closed torsions, $T=-\tilde T, dT=0,$ with  exact Lee forms defined by F. Witt in \cite[Proposition~5.5]{Wit}. In fact, these conditions solve automatically the first two Killing spinor equations in type II string theories in dimension eight (see e.g. \cite{GMW,GKMW,Wit}). Unfortunately, the no-go theorems mentioned above lead to the non-existence of compact generalized  integrable $Spin(7)$--structures with non-vanishing torsion \cite[Corollary~5.6]{Wit}.

 Basic examples of compact $Spin(7)$--manifolds with closed torsion are provided with the group manifolds which have flat torsion connection. However, the corresponding Lee form can be  closed but not exact. 
 
 More precisely, it is well known that any compact 8-dimensional Lie group equipped with a biinvariant metric  and a left-invariant $Spin(7)$--structure $\ph$, generating the biinvariant metric, together with the left-invariant flat Cartan connection with closed torsion 3-form $T=-[.,.],$ preserving the $Spin(7)$--structure $\ph$, is an
invariant $Spin(7)$--structure with a closed torsion 3-form, $dT=0$.
 
 We describe here some examples.
 \begin{exam}
 Consider the product of two copies of the primary Hopf surface, the compact Lie goup $G=S^1\times S^3\times S^3\times S^1=U(1)\times SU(2)\times SU(2)\times U(1)$ with the biinvariant metric and a left-invariant  $Spin(7)$--structure $\ph$, which is parallel with respect to the flat left-invariant  Cartan  connection with closed  torsion $T=-[.,.]$.

More precisely, on the group $G=U(1)\times SU(2)\times SU(2)\times U(1)$ with Lie algebra $g=\mathbb R\oplus su(2)\oplus su(2)\oplus\mathbb R$ and structure equations   
\begin{equation}\label{struc} de_0=0, \quad de_1=e_{23},\quad de_2=e_{31},\quad de_3=e_{12},\quad de_4=e_{56},\quad de_5=e_{64},\quad de_6=e_{45}, \quad de_7=0,
\end{equation}
one considers the  family of left-invariant $SU(3)$ structures $(F,\ps_t,\sp_t)$ on $su(2)\oplus su(2)$ defined by
\[F=e_{14}+e_{25}-e_{36},\quad \ps_t=\cos t\ps+\sin t\sp, \quad \sp_t=-\sin t\ps+\cos t\sp,\quad  where\]
\[\ps=e_{123}+e_{156}-e_{246}-e_{345},\quad \sp=e_{456}+e_{234}-e_{135}-e_{126}, \]
and the family of $G_2$ structures on $G=SU(2)\times SU(2)\times S^1$ defined by
\[\p_t=F \wedge e_7+\ps_t.\]
It is shown in \cite[Proposition~6.2]{FMR} that the left-invariant $G_2$ structures $\p_0, \p_{\frac{\pi}4}, \p_{\frac{3\pi}4}$ on $SU(2)\times SU(2)\times U(1)$ are integrable, $d*\p=\theta\wedge*\p$, induce the same biinvariant metric on the group   $SU(2)\times SU(2)\times U(1)$, have the same closed and co-closed torsion $T=e_{123}+e_{456}$, which is the product of the 3-forms $T=-g([.,.],.)$ of each factor $SU(2)$. The  connection is the left-invariant Cartan connection with torsion $T=-[.,.].$ Moreover, $\p_{\frac{3\pi}4}$ is strongly integrable,
$d\p_{\frac{3\pi}4}\wedge\p_{\frac{3\pi}4}=0,$ and has closed Lee form, 
$\theta_{\frac{3\pi}4}=de_7$. The structure $\p_0$ is of constant type  $d\p_0\wedge\p_0=7vol$ with closed Lee form $\theta_0=de_7$ and  $\p_{\frac{\pi}4}$ is balanced and of constant type, $\delta\p_{\frac{\pi}4}=0,$ and $d\p_{\frac{3\pi}4}\wedge\p_{\frac{3\pi}4}=\frac{11}{\sqrt{2}}vol$.

According to \cite[Theorem~5.1]{II}, the left-invariant $Spin(7)$--structures on  the group $G=SU(2)\times SU(2)\times U(1)\times U(1)$  defined by
\[\ph_t=e_0\wedge\p_t+*^{7}\p_t \quad \textnormal{for} \quad  t\in\{0,\frac{\pi}4,\frac{3\pi}4\},\]
where $*^{7}$ is the Hodge star operator on the seven dimensional factor   $SU(2)\times SU(2)\times U(1)$,  induce the same biinvariant metric on the group   $G$, have the same closed and co-closed torsion $T=e_{123}+e_{456}$, which is the product of the 3-forms $T=-g([.,.],.)$ of each factor $SU(2)$. The $Spin(7)$-torsion connection is the left-invariant Cartan connection with torsion $T=-[.,.]$ and closed Lee form $\theta_t=\frac67e_7+\frac17(d\p_t,*^{7}\p_t)e_0$.
\end{exam}

\begin{exam} Consider again the compact Lie group $G=U(1)\times SU(2)\times SU(2)\times U(1)$. 
\begin{itemize}
\item[a) ]If one takes the left-invariant $Spin(7)$--structure $\ph$ on  the group $G$ defined by  \eqref{s1} it is easy to get using \eqref{struc} that it is  with closed torsion 3-form $T=e_{123}+e_{456}$ and non-closed Lee form $\theta_{\ph}=\frac67(e_4-e_3)$ with $d\theta_{\ph}=\frac67(e_{56}-e_{12})\in spin(7)\cong\Lambda^2_{21},$ according to Theorem~\ref{closT}. 
\item[b)] Define on the group $G$ the three 2-forms
\[F_1=e_{01}+e_{23}+e_{45}+e_{67},\quad F_2=e_{02}-e_{13}+e_{46}-e_{57},\quad F_3=e_{03}+e_{12}+e_{47}+e_{56}\]
and the left-invariant $Spin(7)$--structure $\ph=\frac12(F_1\wedge F_1+F_2\wedge F_2-F_3\wedge F_3)$.  It is easy to verify applying \eqref{struc} that $\ph$ has    closed torsion 3-form $T=e_{123}+e_{456}$ and  closed Lee form $\theta_{\phi}=\frac67(e_7-e_0)$. 
\end{itemize}
\end{exam}

\begin{exam} (The simple Lie group $SU(3)$) Consider $SU(3)$ with a basis of left-invariant 1-forms $e_0,\dots,e_7$, satisfying the structure equations \cite[Chapter~4.1]{WP}
\begin{equation*}
\begin{split}
de_1=de_2=de_4=0,\quad de_3=-e_{12}, \quad de_5=-\frac12e_{34},\quad de_6=\frac12e_{15}-\frac12e_{24},\\
de_7=-\frac12e_{14}-\frac12e_{25}+\frac12e_{36}, \qquad de_0=-\frac{\sqrt{3}}2\Big(e_{45}+e_{67}\Big).
\end{split}
\end{equation*}
The left-invariant $Spin(7)$--structure $\ph$ defined by \eqref{s1} generates the biinvariant metric $g=\sum_{s=0}^7e_s^2$. The left-invariant flat Cartan connection with skew-symmetric torsion $T=-g([.,.],.)$ preserves the  $Spin(7)$--structure $\ph$ and has closed torsion, $dT=0$.  
\end{exam}

\section{Concluding remarks and further investigations}

We have investigated the curvature properties of the unique metric connection with skew-symmetric torsion $T$ preserving the $Spin(7)$--structure on a compact $Spin(7)$--manifold and  relations with the properties of the exterior differential and co-differential of the 3-form torsion. We have obtained precise conditions  on the torsion when the curvature of the torsion connection resembles some of the  curvature properties of the Levi-Civita connection of a $Spin(7)$--manifold with a Riemannian holonomy inside $Spin(7)$. We have expressed the closedness of the  3-form torsion in terms of the Ricci tensor of the torsion connection. 

We have shown that any compact $Spin(7)$--manifold with closed torsion 3-form is a generalized gradient Ricci soliton and this is equivalent to a certain vector field to be parallel with respect to the torsion connection. In particular, this vector field preserves 
the $Spin(7)$--structure. 

An open problem is whether it will be possible to find a type of a $Spin(7)$--Ricci like flow in terms of the given $Spin(7)$--structure on a compact manifold with closed torsion which preserves the closedness of the torsion.

We have obtained that if the curvature of the $Spin(7)$--torsion  connection satisfies the Riemannian first Bianchi identity then it is a fixed point of the  generalized Ricci flow provided the exterior derivative of the Lee form belongs to the Lie algebra $spin(7)$, i.e. the torsion is harmonic and the Ricci tensor of the $Spin(7)$--torsion  connection vanishes. 

In the irreducible case, such a Spin(7)--manifold has to be a simple compact Lie group $SU(3)$ with its biinvariant flat metric and left--invariant $Spin(7)$--structure. 

An  open problem is whether the irreducibility assumption can be dropped, namely,   show that a compact $Spin(7)$--space with the exterior derivative of the Lee form belonging to the Lie algebra $spin(7)$ and curvature of the torsion connection satisfying the Riemannian first Bianchi identity  has to be flat.

Moreover, to the best of our knowledge, there are not known  compact  $Spin(7)$--manifolds with closed torsion which are not  group manifolds and it is an open problem  whether such  spaces do exist.

\end{document}